\documentclass[a4paper,12pt]{article}
\usepackage{amsmath}
\usepackage{amsfonts}
\usepackage{mathrsfs}

\usepackage{amssymb}
\usepackage{latexsym}
\usepackage{amsthm}
\usepackage[dvips]{graphicx}

\usepackage{times,txfonts}
\usepackage{authblk}

\theoremstyle{plain}

\newtheorem{theorem}{Theorem}[section]
\newtheorem{corollary}[theorem]{Corollary}
\newtheorem{lemma}[theorem]{Lemma}
\newtheorem{proposition}[theorem]{Proposition}

\theoremstyle{definition}

\theoremstyle{remark}

\newtheorem{example}[theorem]{Example}
\newtheorem{remark}[theorem]{Remark}

\newtheorem*{claim}{Claim}

\numberwithin{equation}{section}

\def\sgn{\operatorname{sgn}}
\def\Sym{\mathfrak{S}}
\def\Pf{\operatorname{Pf}}
\def\Fam{\mathfrak{F}}
\def\Mat{\mathcal{M}}
\def\wt{\omega}
\def\Det{\operatorname{\det}}

\def\odd#1{\overline{#1}\,}
\def\even#1{\underline{#1}\,}

\begin{document}
\title{Determinant structure for $\tau$-function of 
holonomic deformation of linear differential equations}

\author[1]{Masao Ishikawa}
\author[2]{Toshiyuki Mano}
\author[3]{Teruhisa Tsuda}
\affil[1]{\small Department of Mathematics,
Okayama University,
Okayama 700-8530, Japan.}
\affil[2]{\small Department of Mathematical Sciences,
University of the Ryukyus,
Okinawa 903-0213, Japan.}
\affil[3]{\small Department of Economics,
Hitotsubashi University, 
Tokyo 186-8601, Japan.}
\affil[3]{E-mail: tudateru@econ.hit-u.ac.jp}
\date{October 17, 2017; Revised June 20, 2018}

\normalsize

\maketitle
\begin{abstract}
 In our previous works \cite{MT1, MT2},
 a relationship between Hermite's two approximation problems and 
 Schlesinger transformations 
 of linear differential equations
 has been clarified. 
 In this paper,
 we study $\tau$-functions 
 associated with 
 holonomic deformations 
 of linear differential equations
 by using
 Hermite's two approximation problems.
 As a result, we present a determinant formula for 
 the ratio of $\tau$-functions ($\tau$-quotient).
\end{abstract}

\renewcommand{\thefootnote}{\fnsymbol{footnote}}


\section{Introduction} \label{sec:intro}
There are many results concerning determinant formulas for solutions to the Painlev\'e equations;
see \cite{JKM1,JKM2,KMO,Mas,Tsu1, Tsu2} and references therein.
After pioneering works by D. Chudnovsky and G. Chudnovsky \cite{CC1,CC2},
an underlying relationship between the theory of 
rational approximation for functions 
and
the Painlev\'e equations has been 
clarified 
by several authors
 \cite{Mag,Man,MT1,MT2, Yam}.
This relationship provides a natural explanation 
for the determinant structure of solutions to the Painlev\'e equations.

Among them, the second and third authors of this paper 
studied 
the relationship between two approximation problems by Hermite
(i.e. the Hermite--Pad\'e approximation and the simultaneous Pad\'e approximation)
and isomonodromic deformations of Fuchsian linear differential equations.
They constructed a class of Schlesinger transformations for Fuchsian linear differential equations
using Hermite's two approximation problems
and a duality between them.
As an application, they obtained particular solutions written in terms of iterated hypergeometric integrals 
to the higher-dimensional Hamiltonian systems
of Painlev\'e type 
(that were 
introduced in \cite{Tsu3}).
For details refer to \cite{MT1,MT2}.

In the present paper, 
we study using Hermite's two approximation problems the determinant structure for $\tau$-functions of holonomic deformations
of linear differential equations which 
have regular or irregular singularities of arbitrary Poincar\'e rank.
The main theorem (Theorem \ref{thm:mainthm}) is stated as follows:
fix an integer $L \geq 2$ and 
consider a system of linear differential equations of rank $L$
\begin{equation} \label{eq:intro1}
 \frac{dY}{dx}=\left(
 \sum_{\mu=1}^N\sum_{j=0}^{r_{\mu}}A_{\mu,-j}(x-a_{\mu})^{-j-1}-\sum_{j=1}^{r_{\infty}}A_{\infty,-j}x^{j-1}
 \right)Y,
\end{equation}
where $A_{\mu,-j}$ and $A_{\infty,-j}$ are $L\times L$ matrices independent of $x$.
Let $\tau _0$ be {\it Jimbo--Miwa--Ueno's $\tau$-function} 
(see (\ref{eq:defofomega}) and (\ref{eq:defoftau}))
associated with a holonomic deformation of (\ref{eq:intro1}).
We apply the Schlesinger transformation to (\ref{eq:intro1}) that shifts the 
characteristic exponents at $x=\infty$ 
by 
\[
    {\boldsymbol n}=((L-1)n,-n,\dots,-n)
  \in {\mathbb Z}^L
\]
for a positive integer $n$.
Let $\tau_n$ denote the $\tau$-function associated with the resulting equation.
Then the ratio $\tau_n/\tau_0$ 
($\tau$-quotient) admits a representation in terms of an $(L-1)n\times (L-1)n$ 
block Toeplitz determinant:
\begin{equation}
\label{eq:mainresult}
 \frac{ \tau_n}{\tau_0}=\mbox{const.} D_n,
 \quad
  D_n=\begin{vmatrix} B^{1}_n((L-1)n,n) & \cdots & B^{L-1}_n((L-1)n,n) \end{vmatrix}
\end{equation}
with $B_m^{i}(k,l)$ being a $k\times l$ rectangular Toeplitz matrix 
(see (\ref{eq:Toeplitz}))
whose entries are 
specified by
the asymptotic expansion 
of a fundamental system of solutions to (\ref{eq:intro1})
around $x=\infty$.
It should be noted that
our result is valid for general solutions not only for particular solutions 
such as rational solutions or Riccati solutions.

This paper is 
organized as follows.
In Section~\ref{sec:HPandSP},
we review Hermite's two approximation problems and a certain duality between them.
This duality due to Mahler \cite{Mah}
will be a key point for the construction of Schlesinger transformations in a later section.
We remark that the normalization in this paper is slightly different from that in 
the previous ones \cite{MT1, MT2}.
Therefore, 
we formulate the two approximation problems in a form suitable to
the present case.
In Section~\ref{sec:detPade}, we give determinant representations for the approximation polynomials
and the remainder of the approximation problems.
In our method, these representations turn out to be the nature of the determinant structure of the $\tau$-quotient.
In Section~\ref{sec:holonomic}, we briefly review the theory of holonomic deformation of a linear differential equation 
following \cite{Jimbo-Miwa, Jimbo-Miwa-Ueno}. 
In Section~\ref{sec:Stransformation}, we construct the Schlesinger transformations of linear differential equations
by applying the approximation problems. 
Section~\ref{sec:dettau} is the main part of this paper.
We present the determinant formula for the $\tau$-quotient (see (\ref{eq:mainresult}) or Theorem~\ref{thm:mainthm})
based on the coincidence between 
the Schlesinger transformations and the approximation problems.
A certain determinant identity (see (\ref{eq:key})) plays a crucial role in the proof. 
Section~\ref{sec:particularsol} is devoted to an application of our result.
We demonstrate how to construct particular solutions to the 
holonomic deformation equations such as the Painlev\'e equations. 
We then find some inclusion relations among solutions to holonomic deformations
and, typically, obtain a natural understanding of the determinant formulas for hypergeometric solutions to holonomic deformations.
In Appendix \ref{secA:proofofdet}, we give a proof of the determinant identity 
applied in the proof of Theorem~\ref{thm:mainthm}.
Though this determinant identity can be proved directly,
we will prove its Pfaffian analogue in a general setting 
and then reduce it to the determinant case
in order to simplify the proof and to enjoy better perspectives.

\paragraph{\it Acknowledgement.}
This work was supported 
by a grant-in-aid from the Japan Society for the Promotion of Science
(Grant Numbers 16K05068, 17K05270, 17K05335,
25800082 and 25870234).

\section{Hermite--Pad\'e approximation and simultaneous Pad\'e approximation} 

\label{sec:HPandSP}

In this section, we review Hermite's two approximation problems in a suitable form,
which will be utilized to 
construct Schlesinger transformations for linear
differential equations in a later section.

Let $L$ be an integer larger than one.
Given a set of $L$ formal power series 
\[
 f_0(w), f_1(w), \dots , f_{L-1}(w)\in {\mathbb C}[\![w]\!]
\]
with the conditions
\begin{equation}  \label{eq:fcond}
   f_0(0)=1,\quad f_i(0)=0 \quad (i \neq 0)
\end{equation}
the {\it Hermite--Pad\'e approximation} is formulated as follows:
find 
polynomials
\[
  Q^{(i)}_j(w)\in {\mathbb C}[w]\quad
  (0\leq i,j \leq L-1)
\]
such that
\begin{align}
 &\deg Q^{(i)}_j(w)\leq n-1+\delta_{i,j},  
 \label{eq:degree} 
 \\
 &Q^{(i)}_i(w)f_i(w)+\sum_{j\neq i}wQ^{(i)}_j(w)f_j(w)=w^{Ln}(\delta_{i,0}+O(w)), 
  \label{eq:HPkinji}
 \\
  &Q^{(i)}_i(0)=1 \quad (i \neq 0). 
  \label{eq:seikika} 
\end{align}
There exists a unique set of polynomials $\{Q^{(i)}_j(w)\}$
under a certain generic condition on the coefficients of $f_i(w)$.
The precise condition will be later stated in terms of
non-vanishing  of some block Toeplitz determinants;
see (\ref{eq:unique}) in Section~\ref{sec:detPade}.

In turn, the {\it simultaneous Pad\'e approximation} is formulated as follows:
find 
polynomials
\[
  P^{(i)}_j(w)\in {\mathbb C}[w]
  \quad
   (0\leq i,j \leq L-1)
\]
such that
\begin{align} 
 &\deg P^{(i)}_j(w)\leq n(L-1)-1+\delta_{i,j}, 
 \label{eq:sP1}
 \\
 &f_0(w)P^{(i)}_j(w)-f_j(w)w^{1-\delta_{i,j}}P^{(i)}_0(w)=O(w^{nL}).
    \label{eq:sP2}
\end{align}
Under the same generic condition as above, for each $i$ the polynomials 
$P^{(i)}_j(w)$ $(0 \leq j \leq L-1)$
are uniquely determined up to simultaneous multiplication by constants.

Interestingly enough, these two approximations are in a dual relation;
cf. \cite{Mah}.

\begin{theorem}[Mahler's duality] \label{th:mahler}
  Let $\{Q^{(i)}_j(w)\}$ and $\{P^{(i)}_j(w)\}$ be the Hermite--Pad\'e approximant and the simultaneous Pad\'e approximant, 
  respectively.
 Define  $L\times L$ matrices
 $Q(w)$ and $P(w)$ by
 \begin{gather*}
  Q(w)=\left(w^{1-\delta_{i,j}}Q^{(i)}_j(w)\right)_{0 \leq i,j \leq L-1}
  \in 
  {{\mathbb C}[w]}^{L\times L},
   \\ 
  P(w)=\left(w^{1-\delta_{i,j}}P^{(i)}_j(w)\right)_{0 \leq i,j \leq L-1}
  \in
  {{\mathbb C}[w]}^{L\times L}.
 \end{gather*}
 Then it holds that
 \[
   Q(w){}^{\rm T} P(w)=w^{nL}\cdot D,
 \]
 where $D$ is a diagonal matrix independent of $w$.
\end{theorem}

\begin{proof}
 This can be proved in a procedure similar to Theorem 1.3 in \cite{MT2}.
\end{proof}

We can choose the normalization of $P^{(i)}_j(w)$ such that $D=I$ (the identity matrix).
We will henceforth adopt this normalization.

\section{Determinant representation of Hermite--Pad\'e approximants} \label{sec:detPade}

In this section, we give a concrete description of the solution to the Hermite--Pad\'e approximation problem (\ref{eq:degree})--(\ref{eq:seikika})
 in Section~\ref{sec:HPandSP}.

Without loss of generality, we may
assume $f_0(w)=1$ 
since the approximation conditions remain unchanged
 if we replace 
$\{f_0,f_1,\dots,f_{L-1}\}$ by $\{1,f_1/f_0,\dots,f_{L-1}/f_0\}.$
Therefore, we assume $f_0(w)=1$ in the sequel.
Let us write the power series as
\[
f_i(w)=\displaystyle\sum_{k=0}^{\infty}b^i_kw^k
\quad (0 \leq  i \leq L-1).
\]
Then we see that $b^0_0=1$ and $b^0_k=0$ $(k\neq 0)$ from $f_0(w)=1$
and that $b^i_0=0$ $(i \neq 0)$ from (\ref{eq:fcond}).
Besides
we set $b^i_k=0$ $(k<0)$ for notational convenience.
Let us write the polynomials $Q^{(i)}_j(w)$ as
\[
 Q^{(i)}_j(w)=c^i_{j,0}+c^i_{j,1}w+\cdots +c^i_{j,n-1+\delta_{i,j}}w^{n-1+\delta_{i,j}}
 \quad (0 \leq i,j \leq L-1)
\]
with $c^i_{j,k}$ being the coefficient of $w^k$.
The left-hand side of (\ref{eq:HPkinji})
reads as
\begin{align*}
 Q^{(i)}_if_i+\sum_{j\neq i} w Q^{(i)}_j f_j
 =
 \sum_{k=0}^{\infty}
  \left(
  \sum_{l=0}^nb^i_{k-l}c^i_{i,l}+\sum_{j\neq i}\sum_{l=0}^{n-1}b^j_{k-1-l}c^i_{j,l}
  \right)w^k.
\end{align*}
Hence the approximation condition (\ref{eq:HPkinji})
can be interpreted as a system of linear equations
for the unknowns $c^i_{j,k}$:
\begin{align}
 &\sum_{l=0}^nb^0_{k-l}c^0_{0,l}+\sum_{j\neq 0}\sum_{l=0}^{n-1}b^j_{k-1-l}c^0_{j,l}=0
 \quad
 (0 \leq k \leq Ln-1),
  \label{eq:senkeikinji01} \\
 &\sum_{l=0}^nb^0_{Ln-l}c^0_{0,l}+\sum_{j\neq 0}\sum_{l=0}^{n-1}b^j_{Ln-1-l}c^0_{j,l}=1 \label{eq:senkeikinji02}
\end{align}
for $i=0$;
and 
\begin{equation}
 \sum_{l=0}^nb^i_{k-l}c^i_{i,l}+\sum_{j\neq i}\sum_{l=0}^{n-1}b^j_{k-1-l}c^i_{j,l}=0
 \quad 
 (1 \leq k \leq Ln)
 \label{eq:senkeikinjii}
\end{equation}
for $i \neq 0$.

Let us introduce the column vectors 
\[
 {\boldsymbol c}^i={}^{\rm T}
 \left({\boldsymbol c}^i_0,{\boldsymbol c}^i_1,\dots ,{\boldsymbol c}^i_{L-1}\right)
 \in {\mathbb C}^{Ln+1} \quad
  (0 \leq  i \leq L-1),
\]
where
\[
 {\boldsymbol c}^i_j=
 \left(c^i_{j,0},\dots,c^i_{j,n-1+\delta_{i,j}}
 \right),
\]
and introduce the $k\times l$ rectangular 
Toeplitz matrix
\begin{equation} \label{eq:Toeplitz}
 B_m^{i}(k,l)=\left(b_{m+\alpha-\beta}^{i}\right)_{
 \begin{subarray}{l}
 1\leq \alpha\leq k
 \\
 1\leq \beta\leq l
 \end{subarray}}
 =\begin{pmatrix} 
 b_m^{i} & b_{m-1}^{i} & \cdots & b_{m-l+1}^{i} \\
 b_{m+1}^{i} & b_{m}^{i} & \cdots & b_{m-l+2}^{i} \\
 \vdots & \vdots & \ddots & \vdots \\
 b_{m+k-1}^{i} & b_{m+k-2}^{i} & \cdots & b_{m+k-l}^{i} 
 \end{pmatrix}               
\end{equation}
for the sequence $\{b^i_{j}\}$.
Then 
the linear equations (\ref{eq:senkeikinji01}) and (\ref{eq:senkeikinji02}) are summarized as a matrix form
\begin{equation} \label{eq:renritsu0}
 \mathcal{B}^0{\boldsymbol c}^0={}^{\rm T}(0,\ldots,0,1),
\end{equation}
where $\mathcal{B}^0$ is 
a square matrix of order $Ln+1$
defined by
\begin{equation*}
 \mathcal{B}^0=\begin{pmatrix} B^0_0(Ln+1,n+1) & B^1_{-1}(Ln+1,n) & \cdots & B^{L-1}_{-1}(Ln+1,n) \end{pmatrix}.
\end{equation*}
Similarly,
 (\ref{eq:senkeikinjii}) can be rewritten into
\begin{equation} \label{eq:renritsui}
 \mathcal{B}^i{\boldsymbol c}^i=
  {\boldsymbol 0}=
 {}^{\rm T}(0,\ldots,0)
 \quad (i \neq 0)
\end{equation}
where
$\mathcal{B}^i$ ($i \neq 0$)
are $Ln \times (Ln+1)$  matrices defined by
\begin{equation*}
 \mathcal{B}^i=\begin{pmatrix} B^0_0(Ln,n) & \cdots & B^{i-1}_0(Ln,n) & B^i_1(Ln,n+1) & B^{i+1}_0(Ln,n) & \cdots & B^{L-1}_0(Ln,n) \end{pmatrix}.
\end{equation*}
Solving (\ref{eq:renritsu0}) and (\ref{eq:renritsui}) by Cramer's rule, we have the determinant expressions of the approximants
$Q^{(i)}_j(w)$:
\begin{align*}
 Q^{(0)}_0(w)
 &=
 \frac{1}{ \left| \mathcal{B}^0\right| }
 \begin{vmatrix}  
B^0_0(Ln,n+1) & B^1_{-1}(Ln,n) & \cdots & B^{L-1}_{-1}(Ln,n) \\
                 1,w,\dots,w^n & {\boldsymbol 0} & \cdots & {\boldsymbol 0} \end{vmatrix}, \\
 Q^{(0)}_j(w)
 &=
 \frac{1}{ \left| \mathcal{B}^0\right| }
 \begin{vmatrix}  B^0_0(Ln,n+1) & \cdots & B^j_{-1}(Ln,n) & \cdots & B^{L-1}_{-1}(Ln,n) \\
                 {\boldsymbol 0} & \cdots & 1,w,\dots ,w^{n-1} & \cdots & {\boldsymbol 0} \end{vmatrix} \quad 
(j \neq 0)
\end{align*}
for $i=0$; and 
\begin{equation}
\label{eq:Qij}
Q^{(i)}_j(w)
 =\frac{(-1)^{(L+i)n}}{ \left| \mathcal{B}  \right| }
 \left|
\begin{array}{ccc}
&{\mathcal B}^{(i)}&
\\
{\boldsymbol 0} &
\underbrace{1, w, \ldots, w^{n-1+\delta_{i,j}}}_{\text{\rm $j$th block}}
 &{\boldsymbol 0} 
\end{array}
\right|
\end{equation}
for $i \neq 0$,
where $\mathcal{B}$ is 
a square matrix of order $Ln$ defined by
\[
 \mathcal{B}=\begin{pmatrix} B^0_0(Ln,n) & B^{1}_0(Ln,n) & \cdots & B^{L-1}_0(Ln,n) \end{pmatrix}.
\]
In the latter case we have used the normalization
 (\ref{eq:seikika}).
Note that 
\begin{equation}
\label{eq:unique}
\left|\mathcal{B}^0 \right|\neq 0
\quad \text{and} \quad
 \left| \mathcal{B} \right| \neq 0
\end{equation}
are the conditions for the unique existence of $\{Q^{(i)}_j\}$,
which we will impose throughout this paper.

Next, we concern 
\begin{equation} \label{eq:remind}
 \rho^i(w)=Q^{(i)}_if_i+\sum_{j\neq i}w Q_j^{(i)} f_j  \quad
  (0 \leq i \leq L-1)
\end{equation} 
which are the reminders of the Hermite--Pad\'e approximation problem
(\ref{eq:degree})--(\ref{eq:seikika}).
For $i=0$, we have 
\begin{equation*}
 \rho^0(w)=w^{Ln}(1+O(w)).
\end{equation*}
For $i \neq 0$, 
substituting (\ref{eq:Qij}) shows that
\begin{align*}
\rho^i(w)
&=
\frac{(-1)^{(L+i)n}}{|{\mathcal B} |}
\begin{vmatrix} B^0_0(Ln,n) & \cdots & B^i_1(Ln,n+1) & \cdots & B^{L-1}_0(Ln,n) \\
                   wf_0,\dots,w^nf_0    & \cdots & f_i,wf_i,\dots,w^nf_i & \cdots & wf_{L-1},\dots,w^nf_{L-1} \end{vmatrix}                    
\\
&=O(w^{nL+1}).
\end{align*}
Introduce the {\it block Toeplitz determinants}
\begin{align}
 D_n&=|{\mathcal B}|
 =\begin{vmatrix} B^{0}_0(Ln,n) & \cdots & B^{L-1}_0(Ln,n) \end{vmatrix} 
 \nonumber
 \\
 &=
 \begin{vmatrix} B^{1}_n((L-1)n,n) & \cdots & B^{L-1}_n((L-1)n,n) \end{vmatrix} 
 \label{eq:defDn}
\end{align}
and
\begin{align}
  E_n^{i,j}
  &=
  \begin{vmatrix} B^{0}_0(Ln,n) & \cdots & B^{i}_1(Ln,n+1) & \cdots & B^{L-1}_0(Ln,n) \\
                  B^{0}_{Ln+j-1}(1,n) & \cdots & B^{i}_{Ln+j}(1,n+1) & \cdots & B^{L-1}_{Ln+j-1}(1,n)    
    \end{vmatrix}
\nonumber  
\\
    &=
    \begin{vmatrix} B^{1}_n((L-1)n,n) & \cdots & B^{i}_{n+1}((L-1)n,n+1) & \cdots & B^{L-1}_n((L-1)n,n) \\
                  B^{1}_{Ln+j-1}(1,n) & \cdots & B^{i}_{Ln+j}(1,n+1) & \cdots & B^{L-1}_{Ln+j-1}(1,n)    
    \end{vmatrix},
\label{eq:defE}
\end{align}
where we have used
$b^0_0=1$ and $b^0_k=0$ $(k\neq 0)$.
Thus,
 the coefficients of 
 $\rho^i(w)=w^{Ln}\sum_{j=1}^{\infty}\rho^i_jw^{j}$ are written as
\begin{equation} \label{eq:reprhoij}
 \rho^i_j=
 (-1)^{(L+i)n}
 \frac{E_n^{i,j}}{D_n}.
\end{equation}

\section{Holonomic deformation of a system of linear differential equations} \label{sec:holonomic}
In this section, we briefly review the theory of holonomic deformations of linear differential equations
following \cite{Jimbo-Miwa, Jimbo-Miwa-Ueno}.

We consider 
an $L\times L$ system of linear differential equations
which has regular or  irregular singularities at $x=a_1,\dots,a_N,a_{\infty}=\infty$ on $\mathbb{P}^1$ 
with Poincar\'e rank  $r_{\mu}$ $(\mu=1,\dots,N,\infty)$, respectively:
\begin{equation} \label{lineardiffeq}
 \frac{dY}{dx}=A(x)Y,
\end{equation}
where
\begin{equation*} 
 A(x)=\sum_{\mu=1}^N\sum_{j=0}^{r_{\mu}}A_{\mu,-j}(x-a_{\mu})^{-j-1}-\sum_{j=1}^{r_{\infty}}A_{\infty,-j}x^{j-1}
 \in {\mathbb C}(x)^{L\times L}
\end{equation*}
and $A_{\mu,-j}$ and $A_{\infty,-j}$ are constant matrices independent of $x$.
We assume that 
$A_{\mu,-r_{\mu}}$ $(\mu=1,\dots,N,\infty)$
is diagonalizable as
\begin{equation*}
 A_{\mu,-r_{\mu}}=G^{(\mu)}T^{(\mu)}_{-r_{\mu}}G^{(\mu)-1},
\end{equation*}
where the diagonal matrix
$T^{(\mu)}_{-r_{\mu}}=(t^{(\mu)}_{-r_{\mu}, \alpha} \delta_{\alpha,\beta})_{0 \leq \alpha,\beta \leq L-1}$
satisfies
\begin{align*}
t^{(\mu)}_{-r_{\mu},\alpha}
\neq t^{(\mu)}_{-r_{\mu},\beta}
\quad
 &\text{if} \quad\alpha\neq\beta, \quad r_{\mu}\geq 1, \\
 t^{(\mu)}_{0,\alpha}
  \not\equiv t^{(\mu)}_{0,\beta}
 \mod \mathbb{Z}
 \quad
  &\text{if} \quad 
  \alpha\neq\beta,  \quad  r_{\mu}=0.
\end{align*}
Let us introduce the diagonal matrices
\begin{equation*}
 T^{(\mu)}(x)=
 \left(e^{(\mu)}_{\alpha}(x)\,\delta_{\alpha,\beta}
 \right)_{0 \leq \alpha ,\beta \leq L-1}
\end{equation*}
for $ \mu=1,\ldots,N,\infty$
with
\[
 e^{(\mu)}_{\alpha}(x)=\sum_{j=1}^{r_{\mu}}t^{(\mu)}_{-j, \alpha}
 \frac{ {z_\mu }^{-j}}{-j}+t^{(\mu)}_{0,\alpha}\log z_{\mu},
\quad
 z_{\mu}=
 \left\{\begin{array}{ll} x-a_{\mu} &(1 \leq \mu \leq N) \\ 
x^{-1} & (\mu=\infty). 
\end{array} \right.
\]
Then, 
we can take sectors 
$\mathscr{S}^{(\mu)}_k$ $(1 \leq k \leq 2r_{\mu})$ 
centered at $a_{\mu}$
and
there exists a unique fundamental system of solutions 
to (\ref{lineardiffeq}) having
the asymptotic expansion of the form
\[
Y(x)
\simeq G^{(\mu)}\hat{Y}^{(\mu)}(x)e^{T^{(\mu)}(x)},
\quad
 \hat{Y}^{(\mu)}(x)
 =I+Y^{(\mu)}_1z_{\mu}+Y^{(\mu)}_2{z_{\mu}}^2+\cdots
\]
in each $\mathscr{S}^{(\mu)}_k$.
Note that $\hat{Y}^{(\mu)}(x)$ are  in general divergent
and that even around the same point $z=a_\mu$
these power series in two different sectors may differ by a left multiplication of some constant matrix  
({\it Stokes phenomena}).
Without loss of generality,
 we henceforth assume $G^{(\infty)}=I$.

If we start with the  fundamental system of solutions 
normalized by the asymptotic expansion 
\begin{equation}
\label{eq:fundsol}
Y(x)
\simeq
\hat{Y}^{(\infty)}(x)e^{T^{(\infty)}(x)},
\quad
 \hat{Y}^{(\infty)}(x)
 =I+Y^{(\infty)}_1z_{\infty}+Y^{(\infty)}_2{z_{\infty}}^2+\cdots
\end{equation}
in the sector $\mathscr{S}^{(\infty)}_1$
around $x=\infty$,
then the same solution behaves 
as
\[
 Y(x)
 \simeq G^{(\mu)}\hat{Y}^{(\mu)}(x)e^{T^{(\mu)}(x)}
 {S^{(\mu)}_{k-1}}^{-1} \cdots  {S^{(\mu)}_1}^{-1} 
 C^{(\mu)}
\]
in a different sector $\mathscr{S}^{(\mu)}_k$,
where
$C^{(\mu)}$ and $S^{(\mu)}_j$ are the invertible constant matrices
called the {\it connection matrix} and  {\it Stokes multiplier}, respectively.

We consider a deformation of (\ref{lineardiffeq}) 
by choosing 
$a_1,\ldots,a_N$ and
$t^{(\mu)}_{-j, \alpha}$
$(\mu=1,\dots,N,\infty$; $1 \leq j \leq r_{\mu}$; $0 \leq \alpha \leq L-1)$
as its independent variables 
such that $T^{(\mu)}_0$, $C^{(\mu)}$ and $S^{(\mu)}_j$ are kept invariant. 
Such a deformation is called a {\it holonomic deformation}.
Let $d$ denote the exterior differentiation with respect to
the deformation parameters $\{a_{\mu},t^{(\mu)}_{-j,\alpha}\}$. 
The fundamental system of solutions $Y(x)$ 
specified by (\ref{eq:fundsol}) is subject to the holonomic deformation
if and only if it satisfies
\begin{equation} 
\label{defeq}
 dY(x)=\Omega(x)Y(x),
\end{equation}
where $\Omega(x)$ is a matrix-valued $1$-form
given as
\[
 \Omega(x)=\sum_{\mu=1}^{N}B^{(\mu)}(x)da_{\mu}
 +\sum_{\mu=1,\dots,N,\infty}\sum_{j=1}^{r_{\mu}}\sum_{\alpha=0}^{L-1}B^{(\mu)}_{-j,\alpha}(x)dt^{(\mu)}_{-j,\alpha},
\]
 whose coefficients $B^{(\mu)}(x)$ and $B^{(\mu)}_{-j,\alpha}(x)$ are rational functions in $x$.
From the integrability condition of (\ref{lineardiffeq}) and (\ref{defeq}),
we obtain a system of
nonlinear differential equations 
for $A(x)$ and $G^{(\mu)}$:
\begin{equation} \label{eq:deform}
 dA(x)=\frac{\partial \Omega}{\partial x} (x)+[\Omega(x), A(x)], \quad
 dG^{(\mu)}=\Theta^{(\mu)}G^{(\mu)} 
 \quad 
(1 \leq \mu \leq N).
\end{equation}
We remark that $\Omega(x)$ and $\Theta^{(\mu)}$ are computable from $A(x)$ and $ G^{(\mu)}$ by a rational procedure;
see \cite{Jimbo-Miwa-Ueno} for details.
The $1$-form 
\begin{equation}
\label{eq:defofomega}
 \omega=-\sum_{\mu=1,\dots,N,\infty}
 {\rm tr}\, 
 \underset{x=a_\mu}{\rm Res}
\hat{Y}^{(\mu)}(x)^{-1}\frac{\partial \hat{Y}^{(\mu)}}{\partial x}(x)\,
 dT^{(\mu)}(x)
\end{equation}
is closed, i.e.  $d \omega =0$, for any solution to  (\ref{eq:deform}).
Hence we can define
 the {\it $\tau$-function} 
 $\tau=\tau(\{a_{\mu},t^{(\mu)}_{-j,\alpha}\})$
 by 
 \begin{equation}
 \label{eq:defoftau}
 d\log \tau=\omega.
 \end{equation}

\section{Construction of Schlesinger transformations} 
\label{sec:Stransformation}

In this section, we construct the Schlesinger transformation that shifts the 
characteristic exponents at $x=\infty$
of the system of linear differential equations (\ref{lineardiffeq}) as
\[
  {\boldsymbol t}^{(\infty)}_0=
  \left(
  t^{(\infty)}_{0,0},\dots,t^{(\infty)}_{0,L-1}
  \right )\mapsto {\boldsymbol t}^{(\infty)}_0+{\boldsymbol n},
\]
where 
${\boldsymbol n}=((L-1)n,-n,\dots,-n) \in {\mathbb Z}^L$ and 
$n$ is a positive integer.

Write the power series part of 
$Y(x)
\simeq
\hat{Y}^{(\infty)}(x)e^{T^{(\infty)}(x)}$
 (see (\ref{eq:fundsol}))  as
\begin{equation} 
\label{eq:Y}
 \hat{Y}^{(\infty)}(x)=\Phi (w)
 =\left(  \phi_{i,j}(w)\right)_{0 \leq i,j \leq L-1}, \quad
  \phi_{i,j}(w)=\sum_{k=0}^{\infty}a^{i,j}_k
  w^k,
\end{equation}
where
\[
w=z_{\infty}=\frac{1}{x}.
\]
Namely, $\Phi(w)$ is an $L \times L$ matrix whose entries are formal power series in $w$,
and its constant term is the identity matrix, i.e.
$\phi_{i,j}(0)=\delta_{i,j}$.
Factorizing $\Phi(w)$ into two matrices as
\[
\Phi(w)
=
\left( \frac{ \phi_{i,j}(w)  }{\phi_{j,j}(w)  }\right)_{0 \leq i,j \leq L-1}
\cdot
 \mbox{diag} \left(\phi_{j,j}(w)  \right)_{0 \leq j \leq L-1},
\]
we then define new power series 
$f_i(w)$ from the first column of the former by
\begin{equation} \label{eq:deffi}
 f_i(w)
 =\sum_{k=0}^{\infty}b^i_k w^k
 = \frac{\phi_{i, 0}(w)}{\phi_{0, 0}(w)} 
 \quad
(0\leq   i\leq L-1).
\end{equation}
Note that the coefficients of the {\it diagonal free} part
\[\left( \frac{ \phi_{i,j}(w)  }{\phi_{j,j}(w)  }\right)_{0 \leq i,j \leq L-1}-I\]
can be determined recursively by (\ref{lineardiffeq});
see \cite[Proposition 2.2]{Jimbo-Miwa-Ueno}.
Since it holds that $f_0(w)=1$ and $f_i(0)=0$ $(i \neq 0)$, 
we can apply the Hermite--Pad\'e approximation problem
 (\ref{eq:degree})--(\ref{eq:seikika})
 and the simultaneous Pad\'e approximation problem 
 (\ref{eq:sP1})--(\ref{eq:sP2})
considered in Section~\ref{sec:HPandSP} 
to the set of $L$
 formal power series $\{f_0(w),\dots, f_{L-1}(w)\}$.
Define the matrices
\begin{align*}
Q(w)
&=\left(w^{1-\delta_{i,j}}Q^{(i)}_j(w)\right)_{0 \leq i,j \leq L-1}
\in {{\mathbb C}[w]}^{L \times L},
\\
R(x)&=x^nQ(x^{-1})\in {{\mathbb C}[x]}^{L \times L}.
\end{align*}
Recall here that $\deg Q^{(i)}_j(w) \leq n-1+ \delta_{i,j}$.
The result is stated as follows.

\begin{theorem} 
 The polynomial matrix $R(x)$ provides the representation matrix of the Schlesinger transformation for 
 {\rm(\ref{lineardiffeq})} 
 which shifts the 
 characteristic exponents at $x=\infty$ 
 by ${\boldsymbol n}=((L-1)n,-n,\dots,-n) \in {\mathbb Z}^L$.
\end{theorem}

\begin{proof}
 From Theorem~\ref{th:mahler},
 we have $|Q(w)|\cdot | P(w)|=w^{L^2n}$.
 The conditions for the degrees (\ref{eq:degree}) and (\ref{eq:sP1}) 
 shows that
 $|Q(w)|$ is of degree at most 
  $Ln$ and $|P(w)|$ at most $L(L-1)n$, respectively.
 Consequently, it holds that
  $|Q(w)|=cw^{Ln}$ and $|P(w)|=c^{-1}w^{L(L-1)n}$
 for some constant $c \neq 0$;
 and thus $|R(x)|=c$.
 It implies that $R(x)$ is an invertible matrix 
 at any $x\in{\mathbb C}$.
 Therefore, 
 the transformation $Y(x)\mapsto R(x)Y(x)$ does not affect the regularity or the singularity
 of $Y(x)$ at any $x\in{\mathbb C}$.
 Let us observe the influence at $x=\infty$ of this transformation.
 It follows from the approximation conditions
(\ref{eq:HPkinji}) and (\ref{eq:seikika}) that
 \begin{align}
 \nonumber
  R(x)\Phi(w)&=w^{-n}Q(w)\Phi(w) \\
   \nonumber
  &=w^{-n}
  \left(Q^{(i)}_i\phi_{i, j}+\sum_{k\neq i}wQ^{(i)}_k\phi_{k, j}
  \right)_{0 \leq i,j \leq L-1}
  \\
  &=
  \left(I+O(w)\right)
  \mbox{diag}
  \left(w^{(L-1)n}, w^{-n}, \dots, w^{-n}
  \right).
  \label{eq:henkantenkai}
 \end{align}
 Noticing the expression
 \begin{equation*}
 e^{T^{(\infty)}(x)}=
 \mbox{diag} \left(w^{t^{(\infty)}_{0,j}}  \right)_{0 \leq j \leq L-1}
 e^{\sum_{j=1}^{r_{\infty}}T^{(\infty)}_{-j}\frac{w^{-j}}{-j}}
\end{equation*}
 of the exponential part of $Y(x)$,
 we can conclude that $Y(x)\mapsto R(x)Y(x)$ 
 induces 
 the Schlesinger transformation that shifts the characteristic exponents at $x=\infty$ 
 as
 ${\boldsymbol t}^{(\infty)}_0\mapsto {\boldsymbol t}^{(\infty)}_0+{\boldsymbol n}$.
\end{proof}

\begin{remark}
Taking the determinants of the both sides of (\ref{eq:henkantenkai}), 
 we have $|R(x)| \cdot |\Phi (w)|=1+O(w)$.
 Combining this with $|\Phi (w)|=1+O(w)$ yields $c=| R(x)|=1$.
\end{remark}

\section{Determinant structure of $\tau$-quotients} 
\label{sec:dettau}

In this section, we investigate the influence on the $\tau$-function by the Schlesinger transformation. 


We consider
the Schlesinger transformation 
of 
a linear differential equation (\ref{lineardiffeq}), 
which shifts the characteristic exponents at $x=\infty$ by
\[ {\boldsymbol n}=\left((L-1)n,-n,\dots ,-n\right)\in {\mathbb Z}^L
\]
for a positive integer $n$;
see Section~\ref{sec:Stransformation}.
Let $\tau_n$ denote
the $\tau$-function associated with the holonomic deformation of the resulting linear
 differential equation after the
Schlesinger transformation,
while $\tau_0$ denotes that of the original 
(\ref{lineardiffeq}).

First, 
we shall look at a relation between $\tau_0$ and $\tau_1$.
According to 
\cite[Theorem 4.1]{Jimbo-Miwa}
it holds that
\[
 \frac{ \tau_1}{\tau_0}
  =\mbox{const.}
  \begin{vmatrix} G^{(\infty, \infty)(1,1)}_{1,0} & G^{(\infty,\infty)(1,1)}_{2,0} & \cdots & G^{(\infty, \infty)(1,1)}_{L-1,0} \\
     G^{(\infty, \infty)(1,2)}_{1,0} & G^{(\infty,\infty)(1,2)}_{2,0} & \cdots & G^{(\infty,\infty)(1,2)}_{L-1,0} \\
     \vdots & \vdots & \ddots & \vdots \\
     G^{(\infty,\infty)(1,L-1)}_{1,0} & G^{(\infty,\infty)(1,L-1)}_{2,0} & \cdots & G^{(\infty,\infty)(1,L-1)}_{L-1,0} 
 \end{vmatrix},
\]
where $G^{(\infty,\infty)(1,l)}=(G^{(\infty,\infty)(1,l)}_{i, j})$ $(l\in {\mathbb Z})$ 
is a special case of the {\it characteristic matrices}
and is defined by the following generating function:
\[
  \sum_{l\in{\mathbb Z}}G^{(\infty,\infty)(1,l)}w^l=\Phi(w)
\]
or equivalently
\[
  G^{(\infty,\infty)(1,l)}_{i,j}=a^{i,j}_l
\]
with $a^{i,j}_l=0$ for $l<0$.
Thus we find that
\begin{align} \label{eq:tau1/tau0}
  \frac{\tau_1}{\tau_0}
  =\mbox{const.}\begin{vmatrix} a^{1,0}_1 & a^{2,0}_1 & \cdots & a^{L-1,0}_1 \\
     a^{1,0}_2 & a^{2,0}_2 & \cdots & a^{L-1,0}_2 \\
     \vdots & \vdots & \ddots & \vdots \\
     a^{1,0}_{L-1} & a^{2,0}_{L-1} & \cdots & a^{L-1,0}_{L-1}
 \end{vmatrix}.
\end{align}
Here we note the following elementary fact.
\begin{lemma} 
  \label{lem:simple}
 Let
 \begin{align*}
  \sum_{k=1}^{\infty}\alpha^{i}_k w^k, \quad
   \sum_{k=1}^{\infty}\beta^{i}_k w^k \quad 
   (1 \leq  i \leq L-1) 
   \quad  \text{and} \quad  \sum_{k=0}^{\infty}\gamma_kw^k
 \end{align*}
 be formal power series,
 where $\gamma_0=1$.
 If the relation
  \begin{equation*}
   \sum_{k=1}^{\infty}\alpha^{i}_k w^k
   =
   \left(\sum_{k=1}^{\infty}\beta^{i}_kw^k\right)
   \left(\sum_{k=0}^{\infty}\gamma_kw^k\right)
   \end{equation*}
 among the formal power series holds for each $i$, then the equality
 \begin{equation*}
  \begin{vmatrix} 
  \alpha^{1}_1 & \alpha^{2}_1 & \cdots & \alpha^{L-1}_1 \\
     \alpha^{1}_2 & \alpha^{2}_2 & \cdots & \alpha^{L-1}_2 \\
     \vdots & \vdots & \ddots & \vdots \\
     \alpha^{1}_{L-1} & \alpha^{2}_{L-1} & \cdots & \alpha^{L-1}_{L-1}
 \end{vmatrix}=\begin{vmatrix} 
  \beta^{1}_1 & \beta^{2}_1 & \cdots & \beta^{L-1}_1 \\
     \beta^{1}_2 & \beta^{2}_2 & \cdots & \beta^{L-1}_2 \\
     \vdots & \vdots & \ddots & \vdots \\
     \beta^{1}_{L-1} & \beta^{2}_{L-1} & \cdots & \beta^{L-1}_{L-1}
 \end{vmatrix}
 \end{equation*}
 regarding their coefficients holds. 
\end{lemma}

\begin{proof}
 It can be verified straightforwardly by using
 $\alpha_k^i=\sum_{l+m=k} \beta_l^i \gamma_m$.
\end{proof}

Returning to our situation, we have
\begin{align} \label{eq:a=b}
   \begin{vmatrix} a^{1,0}_1 & a^{2,0}_1 & \cdots & a^{L-1,0}_1 \\
     a^{1,0}_2 & a^{2,0}_2 & \cdots & a^{L-1,0}_2 \\
     \vdots & \vdots & \ddots & \vdots \\
     a^{1,0}_{L-1} & a^{2,0}_{L-1} & \cdots & a^{L-1,0}_{L-1}
 \end{vmatrix}=\begin{vmatrix} b^{1}_1 & b^{2}_1 & \cdots & b^{L-1}_1 \\
     b^{1}_2 & b^{2}_2 & \cdots & b^{L-1}_2 \\
     \vdots & \vdots & \ddots & \vdots \\
     b^{1}_{L-1} & b^{2}_{L-1} & \cdots & b^{L-1}_{L-1}
 \end{vmatrix}
 \end{align}
 from Lemma \ref{lem:simple} since $b^i_k$ and $a^{i,0}_k$ are mutually related by
(see (\ref{eq:Y}) and (\ref{eq:deffi}))
 \[
 f_i(w)=\sum_{k=0}^{\infty}b^i_kw^k=\frac{\phi_{i,0}(w)}{\phi_{0,0}(w)}=
 \frac{\sum_{k=0}^{\infty}a^{i,0}_kw^k}{1+O(w)}.
 \]
It thus follows from (\ref{eq:tau1/tau0}) that
 \begin{align} \label{eq:tau1/tau0=b}
 \frac{ \tau_1}{\tau_0}=\mbox{const.}\begin{vmatrix} b^{1}_1 & b^{2}_1 & \cdots & b^{L-1}_1 \\
     b^{1}_2 & b^{2}_2 & \cdots & b^{L-1}_2 \\
     \vdots & \vdots & \ddots & \vdots \\
     b^{1}_{L-1} & b^{2}_{L-1} & \cdots & b^{L-1}_{L-1}
 \end{vmatrix}.
\end{align}

Next,
we shall track how the entries of $\Phi(w)$ are changed after the Schlesinger transformation.
Define 
\[
 \overline{\Phi} (w)
 =\left(  \overline{\phi}_{i,j}(w)\right)_{0 \leq i,j \leq L-1}, \quad
  \overline{\phi}_{i,j}(w)=\sum_{k=0}^{\infty} \overline{a}^{i,j}_k
  w^k
\]
by
\[
  R(w)\Phi(w)=\overline{\Phi}(w) \
  \mbox{diag}
  \left( w^{(L-1)n}, w^{-n}, \dots, w^{-n}\right).
\]
In particular, the entry $\overline{\phi}_{i,0}(w)$ 
is obtained from the remainder 
of the Hermite--Pad\'e approximation
as 
(see (\ref{eq:remind}) and (\ref{eq:deffi}))
\begin{align*}
  \overline{\phi}_{i,0}(w)
  &=w^{-Ln}
\left(Q^{(i)}_i(w)\phi _{i,0}(w)+\sum_{j\neq i}wQ^{(i)}_j(w)\phi_{j,0}(w)\right) \\ 
   &=w^{-Ln}\phi_{0,0}(w)\rho^i(w).
\end{align*}
Let
\[
   \overline{f}_{i}(w)=\sum_{k=0}^{\infty}\overline{b}^{i}_k w^k
   =\frac{\overline{\phi}_{i,0}(w)}{\overline{\phi}_{0,0}(w)}
\]
and
 $\rho^i(w)=w^{Ln}\sum_{k=1}^{\infty}\rho^i_k w^{k}$ 
 for $1 \leq i \leq  L-1$
 as 
 in the previous sections.
Namely, 
overlined symbols denote the quantities after the Schlesinger transformation that shifts the characteristic exponents at $x=\infty$ 
as
 ${\boldsymbol t}^{(\infty)}_0\mapsto {\boldsymbol t}^{(\infty)}_0+{\boldsymbol n}$. 
Then, by applying Lemma~\ref{lem:simple} twice, we have
\begin{equation} \label{b=a=rho}
 \begin{vmatrix} \overline{a}^{1,0}_1 & \overline{a}^{2,0}_1 & \cdots & \overline{a}^{L-1,0}_1 \\
     \overline{a}^{1,0}_2 & \overline{a}^{2,0}_2 & \cdots & \overline{a}^{L-1,0}_2 \\
     \vdots & \vdots & \ddots & \vdots \\
     \overline{a}^{1,0}_{L-1} & \overline{a}^{2,0}_{L-1} & \cdots & \overline{a}^{L-1,0}_{L-1}
 \end{vmatrix}
 =
 \begin{vmatrix} \overline{b}^{1}_1 & \overline{b}^{2}_1 & \cdots & \overline{b}^{L-1}_1 \\
     \overline{b}^{1}_2 & \overline{b}^{2}_2 & \cdots & \overline{b}^{L-1}_2 \\
     \vdots & \vdots & \ddots & \vdots \\
     \overline{b}^{1}_{L-1} & \overline{b}^{2}_{L-1} & \cdots & \overline{b}^{L-1}_{L-1}
 \end{vmatrix}
 =
 \begin{vmatrix} \rho^{1}_1 & \rho^{2}_1 & \cdots & \rho^{L-1}_1 \\
     \rho^{1}_2 & \rho^{2}_2 & \cdots & \rho^{L-1}_2 \\
     \vdots & \vdots & \ddots & \vdots \\
     \rho^{1}_{L-1} & \rho^{2}_{L-1} & \cdots & \rho^{L-1}_{L-1}
 \end{vmatrix}.
\end{equation}

Finally, combining (\ref{eq:tau1/tau0=b}) and (\ref{b=a=rho})
yields that
\[
\frac{\tau_{n+1}}{\tau_n} 
 =
 \frac{\overline{\tau}_{1}}{\overline{\tau}_0}
  =\mbox{const.}\begin{vmatrix} \overline{b}^{1}_1 & \overline{b}^{2}_1 & \cdots & \overline{b}^{L-1}_1 \\
     \overline{b}^{1}_2 & \overline{b}^{2}_2 & \cdots & \overline{b}^{L-1}_2 \\
     \vdots & \vdots & \ddots & \vdots \\
     \overline{b}^{1}_{L-1} & \overline{b}^{2}_{L-1} & \cdots & \overline{b}^{L-1}_{L-1}
 \end{vmatrix}=\mbox{const.} \begin{vmatrix} \rho^{1}_1 & \rho^{2}_1 & \cdots & \rho^{L-1}_1 \\
     \rho^{1}_2 & \rho^{2}_2 & \cdots & \rho^{L-1}_2 \\
     \vdots & \vdots & \ddots & \vdots \\
     \rho^{1}_{L-1} & \rho^{2}_{L-1} & \cdots & \rho^{L-1}_{L-1}
 \end{vmatrix}.  
\]
Substituting 
 (\ref{eq:reprhoij}) in the above, we obtain
\begin{equation}
  \frac{\tau_{n+1}}{\tau_{n}}
  =\mbox{const.}\,
  {D_n}^{-L+1}\det (E^{i,j}_n)_{1 \leq i,j\leq  L-1}. 
  \label{eq:t(n+1)/t(n)}
\end{equation}
Now we state the main theorem.

\begin{theorem} 
\label{thm:mainthm}
 Consider a holonomic deformation of {\rm(\ref{lineardiffeq})}.
 Let $\tau_0$ be the $\tau$-function associated with 
 {\rm(\ref{lineardiffeq})} and let
 $\tau_n$ be the $\tau$-function associated with the transformed equation from {\rm(\ref{lineardiffeq})}
 by the Schlesinger transformation that shifts the 
 characteristic exponents at $x=\infty$ by
 \[ {\boldsymbol n}=\left((L-1)n,-n,\dots ,-n\right)\in {\mathbb Z}^L
\]
for a positive integer $n$.
Then the following determinant formula for the $\tau$-quotient
holds{\rm:}
 \begin{equation} \label{eq:taun/tau0}
   \frac{\tau_n}{\tau_0}=\mbox{\rm const.}\,D_n,
 \end{equation}
 where $D_n$ is the block Toeplitz determinant 
 defined by {\rm(\ref{eq:Toeplitz})} and
 {\rm(\ref{eq:defDn})} 
 and its entries $b^i_k$ 
 are specified by 
 {\rm(\ref{eq:Y})}  and  {\rm(\ref{eq:deffi})},
 i.e.
 the asymptotic solution to {\rm(\ref{lineardiffeq})}
 at $x=\infty$.
\end{theorem}

\begin{proof}
We have the equality
 \begin{equation}
 \label{eq:key}
    D_{n+1}{D_n}^{L-2}=\det (E^{i,j}_n)_{1 \leq i,j \leq L-1},
 \end{equation}
 which will be shown in Appendix~\ref{secA:proofofdet}.
 Therefore, (\ref{eq:t(n+1)/t(n)}) implies
 \[ 
   \frac{\tau_{n+1}}{\tau_n}=\mbox{const.}\,  
   \frac{D_{n+1}}{D_n}.
 \]
 It is clear from (\ref{eq:defDn}) and  (\ref{eq:tau1/tau0=b}) that
 $
   \tau_1/\tau_0=\mbox{const.}D_1.
 $
 Hence the theorem is proved.
\end{proof}

\begin{remark}
 In the case of a second-order Fuchsian linear differential equation,
 their isomonodromic deformations are governed by the Garnier systems and 
 the formula (\ref{eq:taun/tau0})
 has been established in \cite{Man}.
\end{remark}

\begin{remark}
  Jimbo and Miwa \cite{Jimbo-Miwa} 
 treat determinant representations of $\tau$-quotients for arbitrary Schlesinger transformations and their matrix entries 
 are written in terms of the characteristic matrices. 
 However, the characteristic matrices themselves are,
 in general, too complicated to compute explicitly.
 On the other hand,  Theorem \ref{thm:mainthm} above gives a much simpler representation of $\tau$-quotients 
 in terms of block Toeplitz determinants, 
 though the Schlesinger transformations are restricted to 
 a specific direction shifting the characteristic exponents at one point
 by ${\boldsymbol n}=((L-1)n,-n,\dots,-n)$.
 Note also that our formula involves only the first column of $\Phi(w)$, 
 where $\Phi(w)$ is the power series part of the asymptotic solution to (\ref{lineardiffeq}).
 It is expected that more general Schlesinger transformations are related to other types of approximation problems beyond Hermite--Pad\'e type.
It would be an interesting problem to explore such relationships.
\end{remark}

\begin{example}
Consider a $2 \times 2$ system of linear differential equations
\begin{equation}
\label{eq:p2_lin}
\frac{dY}{dx}=
\left(
\begin{pmatrix} 1& \\
 &-1
\end{pmatrix}x^2 +
\begin{pmatrix} &u\\
-2  \mu/u & 
\end{pmatrix}x
+\begin{pmatrix} \mu + t/2&  -u \lambda \\
-2 (\lambda \mu+ \theta)/u &-\mu-t/2
\end{pmatrix}
\right) Y
\end{equation}
with an irregular singularity of Poincar\'e rank
$3$ at $x=\infty$.
There exists a  unique fundamental system of solutions having
the asymptotic behavior of the form
\[
Y
\simeq
\Phi e^{T^{(\infty)}},
\quad
 \Phi =\left(\phi_{i,j}(w)\right)_{i,j=0,1}
 =I+O(w)
\]
at $x=\infty$,
where $w=1/x$ and
\[
T^{(\infty)}=
\begin{pmatrix}1 &  \\  & -1
\end{pmatrix}\frac{w^{-3}}{3} 
+\begin{pmatrix}t &  \\  & -t
\end{pmatrix} \frac{w^{-1}}{2} 
+\begin{pmatrix}\theta &  \\  & -\theta
\end{pmatrix}\log w.
\]
It thus follows that
\[f(w)=\sum_{k=1}^\infty b_k w^k = \frac{\phi_{1,0}(w)}{\phi_{0,0}(w)}
= -\frac{\mu}{u} w -\frac{\theta+\lambda \mu}{u} w^2 + \frac{\mu(\mu+t)}{2u}w^3+\cdots.
\]
The holonomic deformation of (\ref{eq:p2_lin})
amounts to its compatibility condition with 
\[
\frac{\partial Y}{\partial t}=
\left(
\begin{pmatrix}
1 &  \\
&-1
\end{pmatrix}\frac{x}{2}+ 
\begin{pmatrix}
 & u/2 \\
-\mu /u&
\end{pmatrix}
\right)Y,
\]
which reads
\[
\frac{d \lambda}{dt}= \lambda^2+\mu+\frac{t}{2}, \quad 
\frac{d \mu}{dt}=-2 \lambda \mu -\theta, \quad 
\frac{d }{dt} \log u = -\lambda;
\]
the first two equations are equivalent to
the Painlev\'e II equation (see \cite[Appendix~C]{Jimbo-Miwa}):
\[\frac{d^2\lambda}{dt^2}= 2 \lambda^3+t \lambda + \alpha, \quad \alpha = \frac{1}{2} -\theta.
\]
In this case, since $L=2$ the block Toeplitz determinat $D_n$ 
(see (\ref{eq:defDn})) reduces to a usual one and 
Theorem~\ref{thm:mainthm} shows that the $\tau$-quotient $\tau_n/\tau_0$ is equal to 
\[
 D_n=\begin{vmatrix}
b_n&b_{n-1}& \cdots & b_1
\\
b_{n+1}& b_n & \cdots & b_2 \\
\vdots& \vdots & \ddots & \vdots \\
b_{2n+1}&b_{2n} &\cdots&b_n
\end{vmatrix}
\]
up to multiplication by constants.
It is interesting to note that if we substitute the rational solution 
$\theta=1/2$, $\lambda=0$ and  $\mu=-t/2$ then $f(w)=\sum_{k=1}^\infty b_k w^k$ can be expressed as a logarithmic derivative of a shifted Airy function ${\rm Ai} (t+w^{-2})$; 
this phenomenon has been studied closely 
in connection with
integrable systems
\cite{IKN,JKM1,KMO} (see also \cite{CC2}).
\end{example}

\section{Particular solutions to holonomic deformation} \label{sec:particularsol}
In this section, 
as an application of results in the previous section,
we present a method for constructing particular solutions to 
holonomic deformation equations such as the Painlev\'e equations.

Consider the $L \times L$ system of linear differential equations 
(\ref{lineardiffeq}).
Take a new point 
$a_{N+1}\in{\mathbb C}\setminus \{a_1,\dots,a_N\}$ 
where (\ref{lineardiffeq}) is non-singular.
The solution 
 (\ref{eq:fundsol})
normalized at $x=\infty$ can be expanded around $x=a_{N+1}$ as follows:
\[
  Y(x)=Y(a_{N+1})
  \Psi(w),
  \quad
  \Psi(w)=Y(a_{N+1})^{-1}
  \sum_{n=0}^{\infty}Y^{(n)}(a_{N+1})\frac{{w}^n}{n!},
\]
where $w=x-a_{N+1}$
and $Y^{(n)}(x)$ denotes the $n$th derivative of $Y(x)$ with respect to $x$.
Write the power series part
$\Psi(w)$
as
\[
 \Psi(w)=
 \left(  \psi_{i,j}(w)\right)_{0 \leq i,j \leq L-1}
\]
and put
\[
  f_i(w)=\frac{\psi_{i,0}(w)}{\psi_{0,0}(w)}
  \quad
   (0 \leq i \leq L-1).
\]
We apply the Hermite--Pad\'e approximation problem
 (\ref{eq:degree})--(\ref{eq:seikika})
to the set of formal power series
$\{f_0=1,f_1,\dots,f_{L-1}\}$,
and introduce the matrices
\begin{align*}
Q(w)&=\left(w^{1-\delta_{i,j}}Q^{(i)}_j(w)\right)_{0 \leq i,j \leq L-1}
\in {\mathbb C}[w]^{L\times L},
\\
 R(x)&=(x-a_{N+1})^{-n}Q(x-a_{N+1})\in {\mathbb C}[(x-a_{N+1})^{-1}]^{L\times L}
\end{align*}
 made from its approximants 
$Q^{(i)}_j(w)$.
Using $R(x)$, we define the rational function matrix
\[
  S(x)=Y(a_{N+1})R(\infty)^{-1}R(x)Y(a_{N+1})^{-1}.
\]
Then $\widetilde{Y}(x)=S(x)Y(x)$ satisfies a system of differential equations of the form
\begin{equation} \label{eq:lineardiffeq2}
 \frac{d\widetilde{Y}}{dx}=
 \left(\sum_{\mu=1}^N\sum_{j=0}^{r_{\mu}}
 \widetilde{A}_{\mu,-j}(x-a_{\mu})^{-j-1}
 -\sum_{j=1}^{r_{\infty}}\widetilde{A}_{\infty,-j}x^{j-1}+\widetilde{A}_{N+1}(x-a_{N+1})^{-1}\right)\widetilde{Y}.
\end{equation}
This means that the transformation $Y(x)\mapsto \widetilde{Y}(x)=S(x)Y(x)$ induces one regular singularity $a_{N+1}$ in (\ref{lineardiffeq}).
It is clear by definition of  $S(x)$
that the characteristic exponents of (\ref{eq:lineardiffeq2}) at 
the additional regular singularity
$x=a_{N+1}$ read ${\boldsymbol n}=((L-1)n,-n,\dots,-n)$.
Furthermore, we see that if $Y(x)$ is subject to 
a holonomic deformation of (\ref{lineardiffeq}),
then $\widetilde{Y}(x)$ is also subject to 
that of (\ref{eq:lineardiffeq2})
since $Y(x)$ and $\widetilde{Y}(x)$ have the same monodromy.
Consequently, at the level of holonomic deformations,
we have a certain {\it inclusion relation} 
between solutions as described below.

Suppose for simplicity that 
(\ref{lineardiffeq}) is Fuchsian, i.e. 
$r_\mu=0$ for any $\mu=1,\ldots,N,\infty$.
One can associate with (\ref{lineardiffeq})
an $(N+1)$-tuple 
\[
M=\{(m_{1,1},m_{1,2},\dots,m_{1,k_1}),
\ldots,(m_{N,1},m_{N,2},\dots,m_{N,k_N}),
(m_{\infty,1},m_{\infty,2},\dots,m_{\infty,k_\infty})
\}
\]
of partitions of $L$,
called the {\it spectral type},
which indicates how the characteristic exponents overlap at each 
of the $N+1$ singularities $x=a_\mu$ 
($\mu=1,\ldots,N,\infty$).
Note that by means of the spectral type
the number of accessary parameters in 
(\ref{lineardiffeq}) 
is estimated at
\[
2+(N-1)L^2- \sum_{i=1, \ldots, N, \infty}
 \sum_{j=1}^{k_i} {m_{i,j}}^2;
\]
see e.g. \cite{oshima}. 
The argument above provides a procedure
to obtain a new system 
(\ref{eq:lineardiffeq2})
of spectral type 
$\widetilde{M}=M \cup (L-1,1)$
from the original system (\ref{lineardiffeq}) of spectral type $M$
while keeping the monodromy.
Therefore,
the general solution to the deformation equation of
(\ref{lineardiffeq}) 
gives rise to a particular solution to 
the deformation equation of (\ref{eq:lineardiffeq2}).
This phenomenon is exemplified by the fact that
the Garnier system in $N+1$ variables includes the Garnier system in $N$ variables as its particular solution;
cf. \cite[Theorem 6.1]{Tsu0}
It is also interesting to mention that
 if the original
(\ref{lineardiffeq}) is {\it rigid},
 i.e. having no accessory parameter
such as Gau\ss's 
hypergeometric equation,
then 
the deformation equation of (\ref{eq:lineardiffeq2}) 
possesses 
a solution written in terms of
 that of the rigid system (\ref{lineardiffeq}) itself.
In this case, our procedure gives a natural interpretation to 
Suzuki's recent work \cite{Su1},
in which 
a list of rigid systems or hypergeometric equations appearing in particular solutions 
to the higher order Painlev\'e equations
is presented.

\begin{example}[Case $L=N=2$]
 Let us consider a $2 \times 2$ Fuchsian system of differential equations 
 \begin{equation} \label{eq:hypergeo}
  \frac{dY}{dx}=\left(\frac{A_0}{x}+\frac{A_1}{x-1}\right)Y
 \end{equation}
with three regular singularities $x=0,1, \infty$, 
whose spectral type is $\{(1,1),(1,1),(1,1)\}$.
 We can assume without loss of generality
 that $ |A_i|=0$ and $A_{\infty}=-A_0-A_1$ is diagonal,
i.e. $A_{\infty}=\mbox{diag}(\kappa_1, \kappa_2)$.
It is well known that the entries of a fundamental system of solutions to (\ref{eq:hypergeo})
 can be written in terms of Gau\ss's hypergeometric function.
 If we take an arbitrary point 
 $t\in\mathbb{C}\setminus\{0,1\}$ 
 and apply the procedure above,
 then we obtain a system of differential equations of the form
 \begin{equation} \label{eq:hypP6}
  \frac{d\widetilde{Y}}{dx}=\left(\frac{\widetilde{A}_0}{x}+\frac{\widetilde{A}_1}{x-1}+\frac{\widetilde{A}_t}{x-t}\right)\widetilde{Y};
 \end{equation}
it is a $2 \times 2$ Fuchsian system with four regular singularities $x=0,1, \infty, t$,
whose spectral type is $\{(1,1),(1,1),(1,1),(1,1)\}$.
We know from the construction that the monodromy of (\ref{eq:hypP6}) is independent of $t$, 
i.e. (\ref{eq:hypP6}) 
 is subject to an isomonodromic deformation with a deformation parameter $t$.
 Thus we can derive a particular solution
 written in terms of Gau\ss's hypergeometric functions
  to the Painlev\'e VI equation with constant parameters
 \[
   \alpha=   \frac{(\theta_{\infty}-1)^2}{2}, \quad
   \beta=\frac{-{\theta_0}^2}{2}, \quad 
   \gamma=\frac{{\theta_1}^2}{2}, \quad
   \delta=\frac{1-4n^2}{2},
 \]
 where $\theta_{\infty}=\kappa_1-\kappa_2$,
 $\theta_i=\mbox{tr}A_i= \mbox{tr} \widetilde{A}_i$ $(i=0,1)$
  and $n\in\mathbb{Z}_{\geq 0}$.
 Refer to  \cite{gausspainleve, Jimbo-Miwa} for the Painlev\'e VI equation.
\end{example}

\appendix
\section{Proof of an identity for determinants} 
\label{secA:proofofdet}
%
%
In this appendix we derive the determinant identity (\ref{eq:key}), 
which is used to verify the main theorem of this paper.
%
We first prove its Pfaffian analogue in a general setting to achieve better perspectives,
and then we reduce it to the determinant case.
%
The reader can refer to 
\cite{IO1} for various Pfaffian identities and their applications.
\par\smallbreak
Let $A$ be a set of alphabets,
which is a totally ordered set.
Let $A^*$ denote the set of words over $A$.
For a word  $I\in A^*$ and its permutation $J$,
$\sgn(I,J)$ denotes the sign of the permutation that converts $I$ into $J$
 if $I$ has no duplicate letter,
 and $0$ otherwise.
Given a word $I=i_1i_2\cdots i_{2n}\in A^*$ of length $ \sharp I= 2n$,
its permutation $J=j_{1}j_{2}\cdots j_{2n}$ 
is called a {\it perfect matching} on $I$
if 
$\sigma(2k-1)<\sigma(2k)$ for $1\leq k\leq n$ and
$\sigma(2k-1)<\sigma(2k+1)$ for $1\leq k\leq n-1$,
where $\sigma \in \Sym_{2n}$
and $j_{1}j_{2}\cdots j_{2n}=i_{\sigma(1)}i_{\sigma(2)}\cdots i_{\sigma(2n)}$.
%
%
This perfect matching  
is designated by the configuration in the $xy$ plane
which
 contains 
 $2n$ vertices $v_{k}=(k,0)$ ($1\leq k\leq2n$) labeled with $i_k$ 
and $n$ arcs above the $x$ axis connecting the vertices $v_{\sigma(2k-1)}$ and $v_{\sigma(2k)}$ ($1\leq k\leq n$).
Let $\Fam(I)$ denote the set of all perfect matchings on $I$.
For a perfect matching $J=j_1j_2\cdots j_{2n}\in\Fam(I)$,
we call
$\Mat(J)=\{(j_{2k-1},j_{2k})\,|\,1\leq k\leq n\}$ 
the set of arcs in $J$.
It is easy to see that
the sign $\sgn(I,J)$ equals $(-1)^c$, where $c$ is the number of crossings of the arcs in the configuration of $J$.
For example, 
the set of perfect matchings on a word 
$I=1234$
reads 
\[\Fam(I)=\{1234,1324,1423\}.\]
If we take a perfect matching $J=1423\in\Fam(I)$ then we have the set of arcs
$\Mat(J)=\{(1,4),(2,3)\}$ and 
 $J$ is designated by the following configuration:
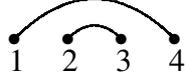
\setlength\unitlength{2pt}
\begin{center}
\begin{picture}(30,20)(0,0)
\put( 0,10){\circle*{2}}
\put(10,10){\circle*{2}}
\put(20,10){\circle*{2}}
\put(30,10){\circle*{2}}
\qbezier( 0,10)(15,25)(30,10)
\qbezier(10,10)(15,15)(20,10)
\put(-1, 4){$1$}
\put( 9, 4){$2$}
\put(19, 4){$3$}
\put(29, 4){$4$}
\end{picture}
\end{center}
There is no crossing of the arcs and 
certainly
$\sgn(I,J)=1$ holds.

Let $f$ be a map which assigns an element of a commutative ring to each pair $(i,j)\in A\times A$ 
such that $f(j,i)=-f(i,j)$.
Such a map is called a {\it skew symmetric} map.
For each perfect matching $J=j_{1}j_{2}\cdots j_{2n}\in\Fam(I)$,
we define the weight $\wt_{f}(J)$ as
\[
\wt_{f}(J)=\sgn(I,J)\prod_{(i,j)\in\Mat(J)}f(i,j).
\]
The {\it Pfaffian} $\Pf_{f}(I)$ of $f$ corresponding to 
the word
$I=i_{1}i_{2}\cdots i_{2n}$ 
is the sum of
the weights $\wt_{f}(J)$, where $J$ runs over all perfect matchings on $I$,
i.e.,
\[
\Pf_f(I)= \sum_{J \in \Fam(I)} \wt_f(J).
\]
We use the convention that $\Pf_f(I)=1$ if $I=\emptyset$.
It is known that 
\begin{equation} \label{eq:indexchange}
\Pf_{f}(K)=\sgn(I,K)\Pf_{f}(I),
\end{equation}
where $K$ is a permutation of $I$.
Especially $\Pf_{f}(I)=0$ if $I$ has a duplicate letter.
For example,
the Pfaffian of $f$ corresponding to $I=1234$
is given as
\[
\Pf_f(I)= f(1,2)f(3,4)-f(1,3)f(2,4)+f(1,4)f(2,3).
\]

The following identity is the Pl\"ucker relation for Pfaffians, which is
originally due to Ohta \cite{O1} and Wenzel \cite{W1}.
Ohta's proof is by algebraic arguments,
and Wenzel employs  the Pfaffian form.
The proof we present here is more combinatorial one based on the same idea as in \cite{IW1}.
\begin{theorem}[cf. \cite{IO1,O1,W1}]
\label{thm:Ohta-Wenzel}
Let $I,J,K\in A^{\ast}$ be words such that $\sharp I$ and $\sharp J$ are odd and $\sharp K$ is even.
Then it holds that
\begin{align}
&\sum_{i\in I}\sgn(IJ,(I\setminus\{i\})iJ)
\Pf_f((I\setminus\{i\})K)\Pf_f(iJK)
\nonumber\\
&\qquad
=\sum_{j\in J}\sgn(IJ,Ij(J\setminus\{j\}))\Pf_f(IjK)\Pf_f((J\setminus\{j\})K).
\label{eq:Ohta-Wenzel}
\end{align}
\end{theorem}
\begin{proof}
%
We put $W_1=KI$, $W_2=JK$ and $W=W_1W_2$.
Let ${\mathfrak G}$ denote the set of perfect matchings on $W$ in which there is exactly
one arc connecting a vertex in $W_1$ and a vertex in $W_2$ and all the other arcs are
between vertices in $W_1$ or between vertices in $W_2$.
%
For example, if $I=123$, $J=456$ and $K=78$
then $W_1=KI=78123$, $W_2=JK=45678$ and $W=W_1W_2=7812345678$.
The following configuration 
designates
such a perfect matching on $W$, $P=7283154867 \in {\mathfrak G}$,
in which the arc $(1,5)$ is the only arc connecting a letter in $W_1$ and a letter in $W_2$:
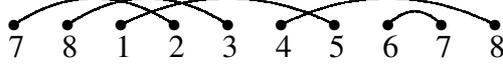
\setlength\unitlength{2pt}
\begin{center}
\begin{picture}(90,20)(0,0)
\put( 0,10){\circle*{2}}
\put(10,10){\circle*{2}}
\put(20,10){\circle*{2}}
\put(30,10){\circle*{2}}
\put(40,10){\circle*{2}}
\put(50,10){\circle*{2}}
\put(60,10){\circle*{2}}
\put(70,10){\circle*{2}}
\put(80,10){\circle*{2}}
\put(90,10){\circle*{2}}
\qbezier( 0,10)(15,20)(30,10)
\qbezier(10,10)(25,20)(40,10)
\qbezier(20,10)(40,20)(60,10)
\qbezier(50,10)(70,20)(90,10)
\qbezier(70,10)(75,15)(80,10)
\put(-1, 4){$7$}
\put( 9, 4){$8$}
\put(19, 4){$1$}
\put(29, 4){$2$}
\put(39, 4){$3$}
\put(49, 4){$4$}
\put(59, 4){$5$}
\put(69, 4){$6$}
\put(79, 4){$7$}
\put(89, 4){$8$}
\end{picture}
\end{center}
For $i\in W_1$ and $j\in W_2$,
let ${\mathfrak G}_{i,j}$ denote the subset of  ${\mathfrak G}$
having the arc $(i,j)$; thereby,  
${\mathfrak G}=\biguplus_{i \in W_1, j \in W_2} {\mathfrak G}_{i,j}$.
%
Let us consider the sums $ \Omega =\sum_{P\in {\mathfrak G}  }\wt_f(P)$
and $\Omega_{i,j}=\sum_{P \in {\mathfrak G}_{i,j}}\wt_f(P)$;
thereby, 
\begin{equation} \label{eq:Omega}
\Omega=\sum_{i \in W_1,j \in W_2}\Omega_{i,j}.
\end{equation}

\begin{claim} It holds that
\begin{equation} \label{eq:claim}
\sum_{j\in W_2}\Omega_{i,j}=\sgn(W,(W_1\setminus\{i\})iW_2)\Pf_f(W_1\setminus\{i\})\Pf_f(iW_2)
\end{equation}
for $i\in W_1$.
\end{claim}

To check the claim, we first associate with each perfect matching 
$P\in{\mathfrak G}_{i,j}$ a pair $(P_1,P_2)$ of perfect matchings such that
$P_1\in\Fam(W_1\setminus\{i\})$ and $P_2\in\Fam(iW_2)$
by shifting $i$ from the original position to the head of $W_2$ in the configuration.
%
For the above example $P \in {\mathfrak G}_{1,5}$
the vertex $1$ is shifted and
the associated pair $(P_1,P_2)$ is thus illustrated as follows:
\setlength\unitlength{2pt}
\begin{center}
\begin{picture}(90,20)(0,0)
\put( 0,10){\circle*{2}}
\put(10,10){\circle*{2}}
\put(20,10){\circle*{2}}
\put(30,10){\circle*{2}}
\put(40,10){\circle*{2}}
\put(50,10){\circle*{2}}
\put(60,10){\circle*{2}}
\put(70,10){\circle*{2}}
\put(80,10){\circle*{2}}
\put(90,10){\circle*{2}}
\qbezier( 0,10)(10,15)(20,10)
\qbezier(10,10)(20,15)(30,10)
\qbezier(40,10)(50,15)(60,10)
\qbezier(50,10)(70,20)(90,10)
\qbezier(70,10)(75,15)(80,10)
\put(-1, 4){$7$}
\put( 9, 4){$8$}
\put(19, 4){$2$}
\put(29, 4){$3$}
\put(39, 4){$1$}
\put(49, 4){$4$}
\put(59, 4){$5$}
\put(69, 4){$6$}
\put(79, 4){$7$}
\put(89, 4){$8$}
\end{picture}
\end{center}
Since $\sgn(W,P)=\sgn(W,(W_1\setminus\{i\})iW_2)\sgn(W_1\setminus\{i\},P_1)\sgn(iW_2,P_2)$,
it is then clear that
\[
\wt_f(P)=\sgn(W,(W_1\setminus\{i\})iW_2)\wt_f(P_1)\wt_f(P_2),
\]
which proves (\ref{eq:claim}).

By the same argument 
we obtain
\begin{equation} \label{eq:claim'}
\sum_{i\in W_1}\Omega_{i,j}=\sgn(W,W_1j(W_2\setminus\{j\}))\Pf_f(W_1j)\Pf_f(W_2\setminus\{j\})
\end{equation}
for $j \in W_2$.
Hence (\ref{eq:Omega}) leads to the desired identity 
(\ref{eq:Ohta-Wenzel})
via (\ref{eq:claim}) and (\ref{eq:claim'}).  
Note that if $i\in K$ or $j\in K$ then there appears a repeated letter in the word,
so we can remove these cases.
\end{proof}
%
%
\begin{corollary}[cf. \cite{IW1,K1}]
\label{cor:Knuth}
Let $I,K\in A^{\ast}$ be words such that $\sharp I$ and $\sharp K$ are even.
Then it holds that 
\begin{equation}
\sum_{{i\in I}\atop{i\neq j}}\sgn(I,(I\setminus\{ i, j \})ij)\Pf_f((I\setminus\{i, j\})K)\Pf_f(ijK)
=\Pf_f(IK)\Pf_f(K)
\label{eq:Knuth}
\end{equation}
for $j \in I$.
\end{corollary}
\begin{proof}
Putting $\sharp J=1$, i.e. $J=j$, in Theorem~\ref{thm:Ohta-Wenzel}
shows that
\[
\sum_{i\in I}\sgn(Ij,(I\setminus\{i\})ij)\Pf_f((I\setminus\{i\})K)\Pf_f(ijK)
=\Pf_f(IjK)\Pf_f(K).
\]
Write $I=i_1i_2\cdots i_n$ and  $I'=i_1\cdots i_{k-1}ji_{k}\cdots i_{n}$.
Then we have
$\sgn(Ij,(I\setminus\{i\})ij)=\sgn(Ij,I')\sgn(I',(I'\setminus\{ i, j \})ij)$
and
$\Pf_f(IjK)=\sgn(Ij,I')\Pf_f(I'K)$ by \eqref{eq:indexchange}.
Hence we obtain
\[
\sum_{ \begin{subarray}{l} i \in I' \\ i\neq j \end{subarray}}
\sgn(I',(I'\setminus\{ i, j \})ij)\Pf_f((I'\setminus\{ i, j \})K)\Pf_f(ijK)
=\Pf_f(I'K)\Pf_f(K),
\]
which coincides with (\ref{eq:Knuth}) if we replace $I'$ with $I$.
\end{proof}
%
%
\begin{corollary}[cf. \cite{K1}]
\label{cor:Knuth2}
Let $I,K\in A^{\ast}$ be words such that $\sharp I$ and $\sharp K$ are
even with $\sharp I=2n$.
Let $F_{f,K}$ be a skew symmetric map on $A \times A$
defined by $F_{f,K}(i,j)=\Pf_{f}(ijK)$.
Then it holds that
\begin{equation}
\Pf_{F_{f,K}}(I)
=\sum_{J\in\Fam(I)}\wt_{F_{f,K}}(J)
=\Pf_{f}(IK)\Pf_f(K)^{n-1}.
\label{eq:Knuth2}
\end{equation}
\end{corollary}

\begin{proof}
Let $I=i_1i_2\cdots i_{2n}$.
We proceed by induction on $n$.
If $n=1$,
it is trivial.
(If $n=2$, \eqref{eq:Knuth2} is implied by Corollary~\ref{cor:Knuth}.)
Assume the $n-1$ case holds for some $n>1$.
In view of
$ \Fam(I)=\bigcup_{k=2}^{2n}  \bigcup_{ J\in\Fam(I\setminus\{i_{1},i_{k}\})  } 
\{i_{1}i_{k}J \}$,
we observe by definition that
\[
\Pf_{F_{f,K}}(I)
=\sum_{k=2}^{2n}\sum_{J\in\Fam(I\setminus\{i_{1},i_{k}\})}
\sgn(I,i_{1}i_{k}J)\,
F_{f,K}(i_{1},i_{k})
\prod_{(i,j)\in\Mat(J)}F_{f,K}(i,j).
\]
Using
$\sgn(I,i_{1}i_{k}J)=\sgn(I,i_{1}i_{k}(I\setminus\{i_{1},i_{k}\}))\sgn(I\setminus\{i_{1},i_{k}\},J)$,
we have
\[
\Pf_{F_{f,K}}(I)
=\sum_{k=2}^{2n}
\sgn(I,i_{1}i_{k}(I\setminus\{i_{1},i_{k}\}))
\Pf_f(i_{1}i_{k}K)
\Pf_{F_{f,K}}(I\setminus\{i_{1},i_{k}\}).
\]
Using the induction hypothesis, we have 
$\Pf_{F_{f,K}}(I\setminus\{i_{1},i_{k}\})=\Pf_{f}((I\setminus\{i_{1},i_{k}\})K)\Pf_{f}(K)^{n-2}$.
By virtue of Corollary~\ref{cor:Knuth},
it is immediate to verify \eqref{eq:Knuth2} for any $n$.
\end{proof}

From here we consider identities for determinants.
Assume the set $A$ of alphabets is a disjoint union of $\overline{A}$ and $\underline{A}$,
i.e.
$A=\overline{A}\uplus\underline{A}$.
Let $R$ and $C$ be any sets of alphabets
which possess injections $R \to\overline A$ and $C \to \underline A$,
denoted by $i\mapsto\overline i$ and $j\mapsto\underline j$, respectively.
For instance,
we let $A=R=C$ be the set of positive integers,
and $\overline{A}$ and $\underline{A}$ the sets of odd and even integers, respectively.
Then we may put $\odd{i}=2i-1$ and $\even{j}=2j$,
which define the injections $R\to\overline A$ and $C\to\underline A$.
For a pair $I=i_1i_2\cdots i_{n}\in R^*$ and $J=j_1j_2\cdots j_{n}\in C^*$ of words of length $n$,
we introduce the word 
\[
m(I,J)=\odd{i_1}\even{j_1}\odd{i_2}\even{j_2}\cdots\odd{i_n}\even{j_n}\in A^*
\]
of length $2n$.
%
%
Let $g$ be a map which assigns an element of a commutative ring to each pair $(i,j)\in R\times C$.
We then define a skew symmetric map $f_g$ on $A \times A$ as follows:
\begin{equation}  \label{eq:fg}
f_{g}(i,j)=\begin{cases}
g(k,l)&\text{ if $i=\odd{k}\in\overline{A}$ and $j=\even{l}\in\underline{A}$,}\\
-g(l,k)&\text{ if $i=\even{k}\in\underline{A}$ and $j=\odd{l}\in\overline{A}$,}\\
0&\text{ otherwise.}
\end{cases}
\end{equation}
We also use the notation
$\Det_g(I,J)$
of determinant
\[
\Det_g(I,J)
=\det(g(i,j))_{i\in I,j\in J}
=\det(g(i_{k},j_{l}))_{1\leq k,l\leq n},
\]
where $I=i_1i_2\cdots i_{n} \in R^*$ and $J=j_1j_2\cdots j_{n} \in C^*$.
%

A determinant can be expressed as a Pfaffian.
%
%
\begin{proposition}[cf. \cite{IO1,K1,M1}]
\label{prop:Pfaffian-determinant}
Let $I\in R^{\ast}$ and $J\in C^{\ast}$ be words such that $\sharp I=\sharp J$.
Then it holds that
\begin{equation*}
\Pf_{f_g}(m(I,J))
=\Det_{g}(I,J).
\end{equation*}
\end{proposition}
\begin{proof}
Let $I=i_{1}i_{2}\cdots i_{n}$ and $J=j_{1}j_{2}\cdots j_{n}$.
To compute $\Pf_{f_g}(m(I,J))$,
we need to consider only perfect matchings 
on $m(I,J)=\odd{i_1}\even{j_1}\odd{i_2}\even{j_2}\cdots\odd{i_n}\even{j_n}\in A^*$ whose arcs are all between $\overline{A}$ and $\underline{A}$;
recall (\ref{eq:fg}).
The set of such perfect matchings is in one-to-one correspondence with $\Sym_n$.
To simplify the description, we first rearrange the word $m(I,J)$ to be 
\[
m(I,J)'=\odd{i_1}\odd{i_2} \cdots\odd{i_n}\even{j_n}\cdots\even{j_2}\even{j_1}
\]
and then consider its perfect matching
$P_\sigma=\odd{i_1}\even{j_{\sigma(1)}}\odd{i_2}\even{j_{\sigma(2)}}\cdots\odd{i_n}\even{j_{\sigma(n)}}$
for each $\sigma\in\Sym_{n}$. 
Because $\sgn(m(I,J),m(I,J)')=1$, we have 
\[
\Pf_{f_g}(m(I,J))=\Pf_{f_g}(m(I,J)')
=\sum_{\sigma \in \Sym_n}  \sgn(m(I,J)',P_\sigma)  \prod_{k=1}^n g(i_k,j_{\sigma(k)})
\]
(see (\ref{eq:indexchange})) and 
\[
\sgn(m(I,J)',P_\sigma) =\sgn(m(I,J),P_\sigma)=\sgn \sigma,
\]
which complete the proof.
\end{proof}

Combining Corollary~\ref{cor:Knuth2}
and
Proposition~\ref{prop:Pfaffian-determinant}
leads to the following determinant identity,
which we may call {\it Sylvester's identity}.
%
%
\begin{corollary}
\label{cor:Knuth2-det}
Let $I,K\in R^{\ast}$ and $J,M\in C^{\ast}$ be words such that $\sharp I=\sharp J=n$ and $\sharp K=\sharp M$.
Let $G_{g,K,M}$ be a map on $R \times C$
defined by
$G_{g,K,M}(i,j)=\Det_{g}(iK,jM)$. 
Then it holds that
\begin{equation}
\Det_{G_{g,K,M}}(I,J)
=\det\left(\Det_{g}(iK,jM)\right)_{i\in I, j\in J}
=\Det_{g}(IK,JM)\Det_{g}(K,M)^{n-1}.
\label{eq:Knuth2-det}
\end{equation}
\end{corollary}

Finally, let us derive the determinant identity (\ref{eq:key}) from Corollary~\ref{cor:Knuth2-det}. 
For notation, recall 
(\ref{eq:Toeplitz}), (\ref{eq:defDn}) and (\ref{eq:defE}).
Let 
$R=C=\{1,2,\dots,(L-1)(n+1)\}$
and put
\[
g(i,j)= b^s_{i-j+s(n+1)} \quad 
\text{with} \quad
s= \left\lfloor  \frac{j}{n+1} \right\rfloor+1
\]
for $(i,j) \in R \times C$,
where $\lfloor x\rfloor$ denotes
 the largest integer which does not exceed $x$. 
We take the words
$I=i_1i_2\cdots i_{L-1} \in R^*$ and 
$J=j_1j_2 \cdots j_{L-1} \in C^*$ of length $L-1$ given by
\[
i_k=(L-1)n+k \quad  \text{and} \quad j_k=(k-1)(n+1)+1 \quad \text{for} \quad 1 \leq k \leq L-1.
\]
Let $[i,j]$ denote the word $i(i+1) \cdots j$ for $i < j$; e.g. $I=[(L-1)n+1,(L-1)(n+1)]$.
We take the words
$K=[1,(L-1)n]\in R^*$ and $M=[1,(L-1)(n+1)] \setminus J \in C^*$
of length $(L-1)n$.
Then it holds that $\Det_g(K,M)=D_n$ and
\begin{align*}
\Det_g(IK,JM)&=(-1)^{\frac{L(L-1)n}{2}}\Det_g([1, (L-1)(n+1)],[1,(L-1)(n+1)] )
\\
&=(-1)^{\frac{L(L-1)n}{2}}D_{n+1}
\end{align*}
since both $IK$ and $JM$ can be rearranged to be  $[1,(L-1)(n+1)]$
and
\[
\sgn(IK,[1,(L-1)(n+1)])\sgn(JM,[1,(L-1)(n+1)])=(-1)^\frac{L(L-1)n}{2}.
\]
In a similar manner, it holds that
\[
\Det_g(i_k K, j_l M)=(-1)^{(L-l)n} E_n^{k,l}
\]
for $1 \leq k, l \leq L-1$.
Hence we obtain (\ref{eq:key}) from 
(\ref{eq:Knuth2-det}) with $n$ replaced by $L-1$.


\end{document}